\documentclass[USenglish,a4paper]{article}

\usepackage[utf8]{inputenc}

\usepackage{amsmath,amsfonts,amsthm}
\usepackage{booktabs}
\usepackage{tcolorbox }

\usepackage{subcaption}
\usepackage{url}
\usepackage{siunitx}
\usepackage{color,soul}

\usepackage{svg}
\usepackage{epstopdf}
\usepackage[percent]{overpic}

\usepackage{authblk}

\usepackage{mathtools,algorithm}
\usepackage{algpseudocode}
\algnewcommand\And{\textbf{and} }
\algnewcommand\Or{\textbf{or} }

\algnewcommand{\Inputs}[1]{%
  \State \textbf{Inputs:}
  \Statex \hspace*{\algorithmicindent}\parbox[t]{.8\linewidth}{\raggedright #1}
}

\DeclareMathOperator*{\argmin}{arg\,min}
\theoremstyle{definition}

\newtheorem{theorem}{Theorem}[section]
\newtheorem{corollary}{Corollary}[theorem]

\newtheorem*{remark}{Remark}

\usepackage{siunitx}
\usepackage{mathtools}

\usepackage{algorithm}
\usepackage{algpseudocode}
\algdef{SE}[DOWHILE]{Do}{doWhile}{\algorithmicdo}[1]{\algorithmicwhile\ #1}%

\newcommand{\rt}{r_{\mathrm{top}}}
\newcommand{\rb}{r_{\mathrm{bot}}}

\newcommand{\R}{\mathbb{R}}
\newcommand{\N}{\mathbb{N}}
\newcommand{\card}{\mathop\mathrm{card}}

\newcommand{\tr}{\mathop\mathrm{tr}}
\newcommand{\tol}{\tau}

\title{Adaptive Gradient Enhanced Gaussian Process Surrogates for Inverse Problems}
\author[1]{Phillip Semler}
\author[1]{Martin Weiser}
\affil[1]{Zuse Institute Berlin, Takustr. 7, 14195 Berlin, Germany}

\begin{document}

\maketitle
\begin{abstract}
Generating simulated training data needed for constructing sufficiently accurate surrogate models to be used for efficient optimization or parameter identification can incur a huge computational effort in the offline phase. We consider a fully adaptive greedy approach to the computational design of experiments problem using gradient-enhanced Gaussian process regression as surrogates. Designs are incrementally defined by solving an optimization problem for accuracy given a certain computational budget. We address not only the choice of evaluation points but also of required simulation accuracy, both of values and gradients of the forward model.
Numerical results show a significant reduction of the computational effort compared to just position-adaptive and static designs as well as a clear benefit of including gradient information into the surrogate training.
\end{abstract}

\section{Introduction}
The response of many parameter-dependent physical models can be described by a functional relation $y$ mapping parameters $p$ to observable quantities $y(p)$. The inverse problem of inferring model parameters $p$ from measurements $y^m$ is frequently encountered in various fields including physics, engineering, finance, and biology.  Point estimates $p_*$, e.g., the maximum likelihood estimate $\mathrm{arg\,min}_p \|y(p)-y^m\|$, are often computed by optimization methods or by sampling the posterior probability distribution of the parameters given the measurement data~\cite{EngelHankeNeubauer1996,KaipioSomersalo2005}.

Often, $y$ is only available as a complex numerical procedure, such as the numerical solution of a partial differential equation. Solving an optimization problem for parameter estimation is then computationally expensive. The effort can be too high for online and real-time applications. Fast surrogate models approximating $y$ are utilized as a replacement for the forward model when solving the inverse problems, in particular when both parameters $p$ and measurements $y^m$ are low-dimensional. Various types of surrogates are employed, including polynomials, sparse grids, tensor trains, artificial neural networks, and Gaussian process regression (GPR)~\cite{Schneider,Zaytsev}, on which we focus here. 

Surrogate models rely on values $y(p_i)$ at specific evaluation points $p_i$  as training data. The quality of the resulting surrogate heavily depends on the number and position of these sample points. Constructing an accurate surrogate model can become computationally expensive when a large number of evaluations is required. Consequently, strategies for selecting near-optimal evaluation points have been proposed, in particular for analytically well-understood GPR~\cite{Rasmussen}. A priori point sets~\cite{Giunta,Queipo} are effectively supplemented by adaptive designs~\cite{Crombecq,Joseph,Lehmensiek,Sugiyama}. Active learning relies on pointwise estimates of the surrogate approximation error for guiding the selection process, usually including the parameter point that maximizes some acquisition function into the training set~\cite{Hennig,Mockus,Wu}. 

When numerical procedures like finite element (FE) solvers are employed to compute training data, the resulting evaluations of $y(p_i)$ are affected by discretization and truncation errors. While uniformly high accuracy can be used, this incurs a high computational effort. The trade-off between accuracy and cost has received limited systematic investigation. Besides two-level approaches using a low-fidelity and a high-fidelity models~\cite{Nitzler}, an adaptive choice of  evaluation tolerances has been proposed~\cite{SagnolHegeWeiser2016}. Significant efficiency gains by joint optimization of evaluation position and tolerance given a limited computational budget have been achieved~\cite{SemlerWeiser}.

In the current paper, we extend~\cite{SemlerWeiser} by including gradient information. Gradients of FE simulations can often be obtained efficiently by solving tangent or adjoint equations~\cite{GriewankWalther2008}.
We make use of gradient enhanced GPR (GEGPR) promising higher accuracy than standard GPR~\cite{Eriksson,Anqi,Wu}. Due to the high information content of gradient data, the required number of evaluation points is small compared to gradient-free GPR. We devise a greedy-type strategy that simultaneously optimizes the next positions and tolerances for forward model and gradient evaluation. For that, we extend the accuracy and work models used in the optimizer to take gradient computation into account.

For measuring the accuracy of the surrogate model, we adopt a goal-oriented approach~\cite{Vexler2013}, assessing the approximation error by its impact on the accuracy of the identified parameter. We aim for a uniform absolute tolerance or, if that is not reachable, at least a uniform bounded deterioration relative to the exact model results. We focus on adaptive FE simulations, where standard a priori error estimates and coarse estimates of the computational work are available. 

The remainder of the paper is structured as follows: In Sec.~\ref{sec:setup} we formalize the parameter identification and surrogate model construction problems, introducing notation and GEGPR. Their adaptive construction is given in Sec.~\ref{sec:doe} after introducing gradient enhanced work and accuracy models. Numerical experiments are provided in Sec.~\ref{sec:experiments}.

\section{Surrogate-based parameter identification} \label{sec:setup}
First, we briefly define the parameter identification  a maximum posterior evaluation in a Bayesian context, before introducing GEGPR surrogate models. For a more detailed discussion, we refer to~\cite{SemlerWeiser}.

\subsection{Inverse problem}
Let the forward problem $y(p) = \mathcal{F}(p):\mathcal{X}\subset \R^d\rightarrow\mathbb{R}^m$ of mapping a parameter vector $p \in \mathcal{X}$ to a model output $y(p) \in \mathbb{R}^m$  be Lipschitz-continuously differentiable and the parameter space $\mathcal{X}$  be bounded and closed. 
We assume that numerical approximations $y_{\tol}(p)$ of $y(p)$ and $y'_{\tol'}(p)$ of the derivative $y'(p)$ can be computed, satisfying the error bounds $\|y_{\tol}(p)-y(p)\| \le \tol$ and $\|y'_{\tol'}(p)-y'(p)\| \le \tol'$, respectively, for any given tolerances $\tol,\tol'>0$ in some norms $\left \| \cdot \right \|$. 
We consider the simplest Bayesian setting of normally distributed measurement errors with likelihood covariance $\Sigma_l\in\R^{m\times m}$ and a Gaussian prior with covariance $\Sigma_p\in\R^{d\times d}$. Given measurements $y^m$, we are interested in the maximum posterior point estimate $p(y^m)$, a minimizer of the negative log-posterior
\begin{equation}\label{eq:objective}
    J(p;y^m) := \frac{1}{2}\| y(p)-y^m \|_{\Sigma_l^{-1}}^2+ \frac{1}{2}  \| p-p^0  \|^2_{\Sigma_p^{-1}} 
\end{equation}
over $p\in\mathcal{X}$, where $\left \| v \right \|_{\Sigma_{}^{-1}} := v^T \Sigma_{}^{-1} v$ for positive symmetric definite $\Sigma$. By regularity of $y$ and compactness of $\mathcal{X}$, $J$ is Lipschitz-continuously differentiable and a, not necessarily unique, minimizer exists. 
For the numerical minimization of~\eqref{eq:objective}, we consider a Gauss-Newton (GN) method~\cite{Deuflhard2004} and assume that the forward model is compatible and the measurement errors sufficiently small such that the residual $y(p(y^m))-y^m$ is small, such that the GN method converges locally.

\subsection{Gradient-enhanced Gaussian process regression} \label{sec:surrogate}

We aim at building a surrogate model $y^*$ for $y:\mathcal{X} \to \R$ based on a set of evaluations of the model and its derivative with certain tolerances $\tol,\tol'$.
We directly consider GEGPR, see~\cite{Wu} and the references therein, and refer to~\cite{Duvenaud,Kuss,Candela,SemlerWeiser} for purely value-based GPR surrogate modeling.

\subsubsection{Designs and training data}

We define $D:=\{\mathcal{D}:\mathcal{X} \to (\R_+\cup \{\infty\})^2 \mid \card X(\mathcal{D})\in\N\}$ as the set of admissible designs, where $X(\mathcal{D}) :=  \{ p \in \mathcal{X} \mid \min_{k\in\{1,2\}}\{\mathcal{D}(p)_k\} < \infty\}$ is the evaluation set. For $p_i\in X(\mathcal{D})$, the design $\mathcal{D}(p_i) = [\tol_i,\tol_i']$ defines the evaluation tolerances for values and derivatives of the forward model $y$ forming the training data $y_i=y_{\tol_i}(p_i)$ and $y_i'=y'_{\tol'_i}(p_i)$. We assume for simplicity that the component-wise evaluation errors $e_i = y_i-y(p_i)$ and $e'_i = y'_i-y'(p_i)$ are independent and normally distributed, i.e. $e_i \sim \mathcal{N}(0,\tol_i^2I_m)$ and $e'_i \sim \mathcal{N}(0,(\tol'_i)^2 I_{md})$, where we denote the identity matrix of dimension $m$ by $I_m$ and identify $e'_i\in\R^{m\times d}$ with its column-major vector representation in $\R^{md}$.

The complete model evaluation at a parameter point $p_i$ thus comprises a vector $y_i\in\R^m$ and a matrix $y_i' \in\R^{m\times d}$. We define $z_i = (y_i,y_i') \in \R^{m+md}$ as the vector representation of the complete evaluation. The training data is then distributed as 
\begin{equation}\label{eq:joint-z-distribution}
    z_i\sim\mathcal{N}(z(p_i),E_i), \quad  
    E_i = \begin{bmatrix} \tol_i I_m \\ & \tol_i' I_{md} \end{bmatrix}.
\end{equation}

\begin{remark}
In principle, the accuracies of computing values $y$ and derivatives $y'$ can be independent. 
In practical computation, however, when evaluating derivatives, obtaining the value at the same position $p_i$ with the same accuracy $\tol_i = \tol_i'$ usually comes for free. We therefore restrict the attention to designs satisfying $\tol_i \le \tol_i'$ for all $i$.
\end{remark}

\subsubsection{Gaussian process prior}

Assume that forward models are a Gaussian process, i.e. model evaluations $z(p_i)$ are jointly normally distributed for any finite evaluation set $X=\{p_i|i=1,\dots, n\}$, which forms the GPR prior distribution
\begin{equation}\label{eq:EGP}
    \pi_{\mathrm{prior}}(Z) = \mathcal{N}(\zeta,\tilde K) 
\end{equation}
for joint model responses $Z = (z(p_i))_{i=1,\dots,n} \in \R^{n(m+md)}$ with given mean $\zeta$. The structure of $Z$ induces a nested block form of the covariance matrix $\tilde K = (\hat K_{ij})_{ij}$, $i,j=1,\dots,n$, see~\cite{Solak,Wu}, 
with 
\begin{equation}\label{eq:kovarianceenhanced}
   \hat K_{ij} =  \begin{bmatrix}
                K_{ij} & K_{ij}' \\ 
                K_{ji}' &  K_{ij}''
                \end{bmatrix} .
\end{equation}
Here, $K_{ij}= k(p_i,p_j) K\in\R^{m\times m}$ is the covariance of the model values $y(p_i)$ and $y(p_j)$, assumed to be separable in terms of a positive definite kernel $k$ depending on the evaluation positions $p_i$ and $p_j$ only, and a matrix $K$ describing the covariance of model response components independently of the evaluation position. Most often, $k$ is assumed to depend on the distance of $p_i$ and $p_j$ only, such that $k(p_i,p_j) = k(\|p_i - p_j\|)$. A common choice is the Gaussian kernel $k(x) = \exp(-\sigma x^2)$. 

The remaining blocks $K_{ij}'\in\R^{m\times dm}$ and $K_{ij}''\in\R^{dm\times dm}$ describe the covariance of model values and model derivatives $y'$, and of the covariance of model derivatives, respectively. Since differentiation is a linear operation, the derivatives are again normally distributed, and the covariances are given in terms of the kernel's derivatives as $\mathrm{cov}(y(p),y'(q)) = \nabla_q k\left(p,q\right)$, see~\cite[Chap.~9.4]{Rasmussen}.
We therefore can write the matrix blocks in terms of the Kronecker product $\otimes$ as
\begin{equation} \label{eq:gradient-correlation}
    K'_{ij} = \partial_{p_j}k(p_i,p_j) \otimes K , 
    \quad K'_{ji} = \partial_{p_i}k(p_i,p_j) \otimes K , 
    \quad K''_{ij} = \partial_{p_i}\partial_{p_j}k(p_i,p_j) \otimes K.
\end{equation}

\subsubsection{Regression}

Given the GPR prior~\eqref{eq:EGP} and approximate complete evaluations $\hat{Z} = (z_i)_{i=1,\dots,n}\in\R^{n(m+md)}$ according to~\eqref{eq:joint-z-distribution}, the GPR likelihood is $\pi_{\rm like}(\hat Z\mid Z) = \mathcal{N}(z|_X,\tilde E)$, with $\tilde E = \mathrm{diag}(E_1,\dots,E_n)$,
and the GPR posterior probability for the complete model response $Z$ is
$
    \pi_{\rm post}(Z\mid\hat Z) \propto \pi_{\rm like}(\hat Z\mid Z) \pi_{\rm prior}(Z)
$
by Bayes' rule. As a product of two Gaussian distributions, $\pi_{\rm post}$ is again Gaussian, with covariance $\Gamma = \left(\tilde K^{-1} + \tilde E^{-1}\right)^{-1}$ and mean $\bar Z = \Gamma(\tilde K^{-1} \zeta + \tilde E^{-1}\hat Z)$. In particular for $p_n$, the marginal covariance $\Gamma_{\mathcal{D}}(p_n;\tol_n,\tol_n')
\in \R^{m(1+d)\times m(1+d)}$ according to the block structure of $Z$ and $\Gamma$ is a monotone function of $\tol_n$ and $\tol_n'$, i.e. $\Gamma_{\mathcal{D}}(p_n,\tol_n+\delta,\tol_n'+\delta') \succeq \Gamma_{\mathcal{D}}(p_n,\tol_n,\tol_n')$ for $\delta,\delta'\ge 0$.

Regression is performed for a point $p_n$ with unknown model response, i.e. formally $\tol_n = \tol'_n = \infty$, by extracting the marginal mean $z_{\mathcal D}(p_n) = \bar Z_n \in \R^{m+md}$ and the marginal covariance $\Gamma_{\mathcal{D}}(p_n)$. Note that $z_{\mathcal D}(p_n) = (y_{\mathcal{D}}, y'_{\mathcal{D}})(p_n)$ provides an estimate for both, values and derivatives of the model at $p_n$, and $\Gamma_{\mathcal{D}}(p_n)$ quantifies the uncertainty of this estimate. Due to the choice~\eqref{eq:gradient-correlation} for the prior covariance of values and gradients, the GPR model is consistent, i.e. the derivative estimate $y_{\mathcal{D}}'$ coincides with the parametric derivative of the value estimate $y_{\mathcal{D}}$.

\begin{remark}
    If the components of the model response are assumed to be uncorrelated, the covariance matrix $K$ is diagonal. With the independence of the evaluation errors postulated in~\eqref{eq:joint-z-distribution}, both $\tilde K$ and $\tilde E$, and therefore $\Gamma$, decompose into $m$ independent blocks, one for each component. Then, the GPR can be performed independently for each component using a simpler scalar GPR model. An appropriate covariance model for the components can, however, result in superior results~\cite{bai2023gaussian}.
\end{remark}

\section{Adaptive Training Data Generation} \label{sec:doe}

Replacing exact model evaluations $y(p)$ in the objective~\eqref{eq:objective} by a cheaper GPR model $y_{\mathcal{D}}(p)$ based on a design $\mathcal{D}$ yields maximum posterior point estimates 
\[
p_{{\mathcal{D}}}(y^m) = \argmin_{p\in \mathcal{X}} J_{{\mathcal{D}}}(p;y^m)
:= \frac{1}{2}\|y_{{\mathcal{D}}}(p)-y^m\|_{\Sigma_l^{-1}}^2 + \frac{1}{2}\|p-p^0\|_{\Sigma_p^{-1}}^2
\]
and saves computational effort when computing Gauss-Newton steps $\Delta p_{\mathcal{D}}(p,y^m)$ by solving
\begin{equation*}
    \left( y_{\mathcal{D}}'(p)^T\Sigma_l^{-1}y_{\mathcal{D}}'(p) + \Sigma_p^{-1} \right) \Delta p_{{\mathcal{D}}} 
    = (y_{\mathcal{D}}(p)')^T\Sigma_l^{-1}(y_{\mathcal{D}}(p)-y^m) + \Sigma_p^{-1}(p^0-p).
\end{equation*}
It also incurs both some error $p_{{\mathcal{D}}}(y^m)-p(y^m)$ of the resulting identified parameters and a considerable computational effort for evaluating the model according to the design ${\mathcal{D}}$ beforehand.

If computational resources are limited, the challenge is to determine which parameters $p_i$ simulations should be run with, and which tolerances $\tol_i,\tol_i'$ should be used, to achieve the best accuracy. This design of experiments problem for $\mathcal{D}$ involves balancing the competing objectives of minimizing the expected approximation error $E(\mathcal{D})$ of the surrogate model and minimizing the computational effort $W(\mathcal{D})$ required to create the training data. Since little is known about the model derivative $y'$, and consequently about $E(\mathcal{D})$, before any simulations have been performed, we follow a sequential design of experiments approach~\cite{SemlerWeiser} by incrementally spending computational budget of size $\Delta W$. In each step, we thus have to solve an incremental design problem for $\tilde{\mathcal{D}}\in D$ refining a given preliminary design $\mathcal{D}$:
\begin{equation}\label{eq:doe-incremental}
    \min_{\tilde{\mathcal{D}}\le \mathcal{D}} E(\tilde{\mathcal{D}}) \quad\text{subject to } W(\tilde{\mathcal{D}}|\mathcal{D}) \le \Delta W.
\end{equation}
We will establish quantitative error estimates $ E(\mathcal{D})$ and work models $W(\mathcal{D})$, and then develop a heuristic for approximately solving problem \eqref{eq:doe-incremental}.

\subsection{Accuracy model}

First we need to quantify the parameter reconstruction error $p_{\mathcal{D}}(y^m)-p(y^m)$ in terms of the measurement error variance $\Sigma_l$ and the surrogate model approximation quality $y_{\mathcal{D}}-y$ depending on the design $\mathcal{D}$.
We start by establishing an estimate of the parameter reconstruction error for deterministic functions $y(p)$ and $y_{\mathcal{D}}(p)$. 

\begin{theorem} \label{th:local-errorbound}
    Assume that $y$ is twice continuously differentiable with uniformly bounded first and second derivatives in a neighborhood $B$ of a minimizer $p^*\in\mathcal{X}$ of $J(p;y^m)$ for some measurement data $y^m$, such that the residual $ \left\| y(p^*)-y^m\right\|_{\Sigma_l^{-1}}$
    is sufficiently small and $y'(p^*)\Sigma_l^{-1}y'(p^*) + \Sigma_p^{-1}$ is positive definite.

    Then, there are $\bar\epsilon,\bar\epsilon'>0$ as well as constants $a,a'$, such that for all $\epsilon\le \bar\epsilon$ and $\epsilon' \le \bar\epsilon'$ and 
    surrogate models $y_\mathcal{D}:\R^d\to\R^m$ with $\|y_{\mathcal{D}}-y\|_{L^\infty(B)} \le \epsilon$ and $\|y'_{\mathcal{D}} - y'\|_{L^\infty(B)} \le \epsilon'$
    there is a locally unique minimizer $p_\mathcal{D}(y^m)$ of $J_\mathcal{D}$ satisfying the error bound
    \begin{equation}\label{eq:error-bound}
        \left\|p_{\mathcal{D}}(y^m)-p^*\right\| \le a\epsilon + a'\epsilon'.
    \end{equation}
\end{theorem}
A more detailed claim and the proof is given in~\cite[Corollary 3.1.1]{SemlerWeiser}. Though the constants $a,a'$ are quantitatively unavailable for concrete problems, Thm.~\ref{th:local-errorbound} establishes a linear relation between the surrogate model accuracy and the incurred error in the parameter estimate. The factors $a$ and $a'$ can be estimated numerically by bounding linearized error transport through a Gauß-Newton iteration. 

We expect a minimal absolute reconstruction error, but due to measurement inaccuracies, we cannot avoid some error $e_0(p)$ and therefore consider a small relative error to be acceptable. The unavoidable error level can be bounded in terms of the problem properties, see~\cite{SemlerWeiser}. We therefore define the local error quantity 
\begin{equation}\label{eq:local-error-model}
e_{\mathcal{D}}(p) := \frac{a(p)\epsilon + a(p)'\epsilon'}{1+\alpha e_0(p)} .
\end{equation}

Since during the construction of the surrogate model $y_{\mathcal{D}}$ by minimizing~\eqref{eq:doe-incremental} the measurement values $y^m$ and hence the parameter position $p=p(y^m)$ of interest are unknown, the error quantity~\eqref{eq:local-error-model} needs to be considered over the whole parameter region $\mathcal{X}$.
We therefore define the accuracy model 
\begin{equation}\label{eq:error-model}
    E({\mathcal{D}}) := \|e_{\mathcal{D}}\|_{L^q(\mathcal{X})} \quad \text{for some $1\le q < \infty$},
\end{equation}
that is to be minimized by selecting an appropriate design ${\mathcal{D}}$. 
Choosing $q\approx 1$ would focus on minimizing the average parameter reconstruction error, while choosing $q$ very large would focus on the worst case and impose a roughly uniform accuracy. For the numerical experiments in Sec.~\ref{sec:experiments} we have chosen $q=2$.

Still missing are the surrogate model error bounds $\epsilon,\epsilon'$ in terms of the design ${\mathcal{D}}$. Unfortunately, for virtually all cases of practical interest, there is little hope for obtaining simultaneously rigorous and quantitatively useful bounds. For GPR surrogate models, the global support of the posterior probability density precludes the existence of a strict bound, though its fast decay provides thresholds that are with high probability not exceeded. The assumed normal distribution of errors $z(p)-z_{\mathcal{D}}(p) \sim \mathcal{N}(0,\Gamma_{\mathcal{D}}(p))$ implies that both the Euclidean norm $\|y(p)-y_{\mathcal{D}}(p)\|_2$ and the Frobenius norm $\|y'(p)-y'_{\mathcal{D}}(p)\|_F$  are generalized-$\chi^2$-distributed and hence formally unbounded. Instead of a strict bound, we use a representative statistical quantity for $\epsilon$ and $\epsilon'$ based on the marginal covariance $\Gamma_{y_{\mathcal{D}}} = \Gamma_{\mathcal{D}}(p)_{1:m,1:m}$, such as the mean~\cite{MathaiProvost1992} 
\begin{equation} \label{eq:GPR-error-estimate}
    \epsilon := \tr(\Gamma_{y_\mathcal{D}}).
\end{equation} 
Analogously, $\epsilon'$ can be defined in terms of $\Gamma_{y'_{\mathrm{D}}}= {\Gamma_{\mathcal{D}}}(p)_{m+1:m(d+1),m+1:m(d+1)}$. Inserting the thus chosen values of $\epsilon$ and $\epsilon'$ into~\eqref{eq:local-error-model} completes the accuracy model.

One advantage of integrating derivative information into the GPR surrogate is that there is an explicit variance estimate available for defining $\epsilon'$. In contrast, GPR surrogates defined only in terms of the values $y$ as considered in~\cite{SemlerWeiser} must rely on an empirical relation of $\epsilon$ and $\epsilon'$.

\subsection{Work model}

The evaluation of the forward model $y(p)$ usually involves some kind of numerical approximation, resulting in an approximation $y_\tol(p)$. While in principle any accuracy $\|y_\tol(p)-y(p)\| \le \tol$ for arbitrary tolerance $\tol>0$ can be achieved, this requires a computational effort $W(\tol)$ to be spent on the evaluation. 
For adaptive finite element computations in $\R^l$ with ansatz order $r$ and $N$ degrees of freedom, we expect a discretization error $\epsilon = \mathcal{O}(N^{-r/l})$~\cite{DeuflhardWeiser2012}. Assuming an optimal solver with computational work $\mathcal{O}(N)$, we obtain a work model $W(\tol) = \tol^{-2s}$ with $s=l/(2r)>0$, see~\cite{Weiser}. $W$ is monotone and satisfies the barrier property $W(\tol)\to\infty$ for $\tol\to 0$ and the minimum effort property $W(\tol)\to 0$ for $\tol\to\infty$.

Including gradient information is of particular interest if derivatives can be computed efficiently, e.g., with adjoint methods, such that the cost of derivative computation is a small multiple $c$ of the value computation cost, and independent of the number $d$ of parameters~\cite{GriewankWalther2008}. This leads to the work model $W(\tol') = c(\tol')^{-2s}$. 

The computational effort incurred by a design $\mathcal{D}$ is then
\begin{equation}\label{eq:work-model}
    W({\mathcal{D}}) := \sum_{p_i\in X(\mathcal{D})} \tol_i^{-2s} + c\; (\tol'_i)^{-2s}.
\end{equation}
Being interested in \emph{incremental designs}, we assume $\mathcal{D}$ to be a design already realized. Evaluating the model on a finer design $\tilde{\mathcal{D}} \le \mathcal{D}$ can consist of simulating the model for parameters $p\not\in X(\mathcal{D})$, or improving the accuracy of already performed simulations for $p\in X(\mathcal{D})$ with $\tilde{\mathcal{D}}(p)_k< \mathcal{D}(p)_k$ for some $k\in\{1,2\}$, or both. If already conducted simulations can be continued instead of started again, the computational effort of obtaining the training data set $\tilde{\mathcal{D}}$ from $\mathcal{D}$ is
\begin{equation} \label{eq:fe-work}
W(\tilde{\mathcal{D}}\mid\mathcal{D}) = W(\tilde{\mathcal{D}}) - W(\mathcal{D}).
\end{equation}

\subsection{The design of computer experiments problem}

The design problem~\eqref{eq:doe-incremental} is combinatorial in nature due to the unknown number $n$ of evaluation points and the decision between introducing new points or reusing existing ones. It is therefore highly nonlinear and difficult to treat rigorously due to the parameter locations $p_i$ to be optimized. In particular, relaxing the design to the space of nonnegative regular Borel measures, as used in~\cite{NeitzelPieperVexlerWalter2019}, is not feasible due to the nonlinearity of the work model~\eqref{eq:fe-work}. 

We therefore take a heuristic two-stage approach. We simplify the problem by decoupling the choice of evaluation positions $p_i$ from the choice of evaluation accuracies $\tol_i, \tol_i'$.
In the first stage, we select few promising additional points for inclusion into the evaluation set $X$. In the second stage, the evaluation tolerances $\tol_i$ and $\tol_i'$ are then optimized for minimal error $E(\tilde{\mathcal{D}})$ given the computational budget constraint $W(\tilde{\mathcal{D}}|\mathcal{D}) \le \Delta W$. Then the necessary evaluations of the forward model are performed and the GP surrogate model is updated, yielding improved error estimates for the next incremental design.

A parameter point $p$ is particularly promising for inclusion, if its predicted impact on the overall error $E$ is large. The impact can be estimated by the local derivative of $e_{\mathcal{D}}(p)^q$ with respect to work spent for approximating $y(p)$. We assume $\tol' = \beta \tol$ for some $\beta\in[0,1]$ and consider
\begin{align*}
    \frac{\partial e_{\mathcal{D}}(p)}{\partial W(p)} (\beta)
    &= \frac{\partial e_{\mathcal{D}}(p)}{\partial (\epsilon,\epsilon')} \, 
      \frac{(\partial\epsilon,\epsilon')}{\partial\tol} 
      \left(\frac{d W}{d \tol}\right)^{-1} \\
    &= \frac{[a(p),a(p)']}{1+\alpha e_0(p)} 
            \left[ \tr \frac{\partial\Gamma_{y_\mathcal{D}}}{\partial (\tol,\tol')} , \tr\frac{\partial\Gamma_{y'_{\mathrm{D}}}}{\partial(\tol,\tol')} \right]
            \begin{bmatrix}1 \\ \beta\end{bmatrix}
             \frac{\tol^{2s+1}}{(-2s)(1+c\beta^{-2s})}.
\end{align*}
Neglecting constant factors, which are irrelevant for the \emph{relative} merit of candidate points, and evaluating the derivatives at the current tolerance level $\tr(\Gamma_{y_\mathcal{D}})$, we define the acquisition function 
\begin{equation}\label{eq:utility}
    \phi(p) = \max_{\beta\in[0,1]} e_{\mathcal{D}}(p)^{q-1} \frac{\partial e_{\mathcal{D}}(p)}{\partial W(p)} (\beta).
\end{equation}
As it is often only practical to compute derivatives along with the values or not at all, we may restrict the choice of $\beta$ to the two extreme cases $\beta\in\{0,1\}$.
Selecting candidate points for inclusion into the valuation set $X$ can be based on finding local maximizers of the acquisition function $\phi$, or finding the best points from a random sampling or from low discrepancy sequences.

With the evaluation set $X$ fixed, the optimization problem~\eqref{eq:doe-incremental} reduces to the nonlinear programming problem $\min_{\tilde \tol,\tilde\tol'}E(\tilde\tol,\tilde\tol')$ for the new evaluation tolerances $\tilde\tol$, $\tilde\tol'$, subject to the pure improvement requirement $\tilde\tol \le \tol$, $\tilde\tol' \le \tol'$ and the computational budget constraint $W(\tilde\tol,\tilde\tol') \le W(\tol,\tol')+\Delta W$. For convexity reasons, we consider an equivalent reformulation in terms of $v=\tol^{-2}$ and $v'=(\tol')^{-2}$:
\begin{equation}\label{eq:doe-incremental-reduced}
\min_{\tilde v,\tilde v'\in\R^n_+} E(\tilde v,\tilde v')^q \quad \text{s.t.} \quad
\tilde v \ge v, \; \tilde v' \ge v', \; \tilde v' \le \tilde v, \; 
W(\tilde v,\tilde v') \le W(v,v')+\Delta W.
\end{equation}

\begin{theorem} \label{th:convex-objective}
    The objective $\tilde{E}(\tilde v, \tilde v')^q$ is convex. 
\end{theorem}
\begin{proof}
    As in~\cite[Thm.~3.2]{SemlerWeiser}, $\epsilon = \tr(\Gamma_{y_{\mathcal{D}}})$ and $\epsilon' = \tr(\Gamma_{y'_{\mathcal{D}}})$ are convex in $v$ and $v'$ at all $p\in\mathcal{X}$, and consequently, $e_{\mathcal{D}}(p)$ is convex as well. Due to convexity and monotonicity, $e_{\mathcal{D}}(p)^q$, and due to linearity, also its integral $E$ are convex.
\end{proof}
\begin{remark}
    In contrast to purely value-based GPR surrogates, here the objective $E^q$ is in general not strictly convex, since the covariance kernel containing entries of the form $\partial_{p_i}k(p_i,p_j)$ is not pointwise positive.
\end{remark}

\begin{figure}
    \centering
    \includegraphics[width=\textwidth]{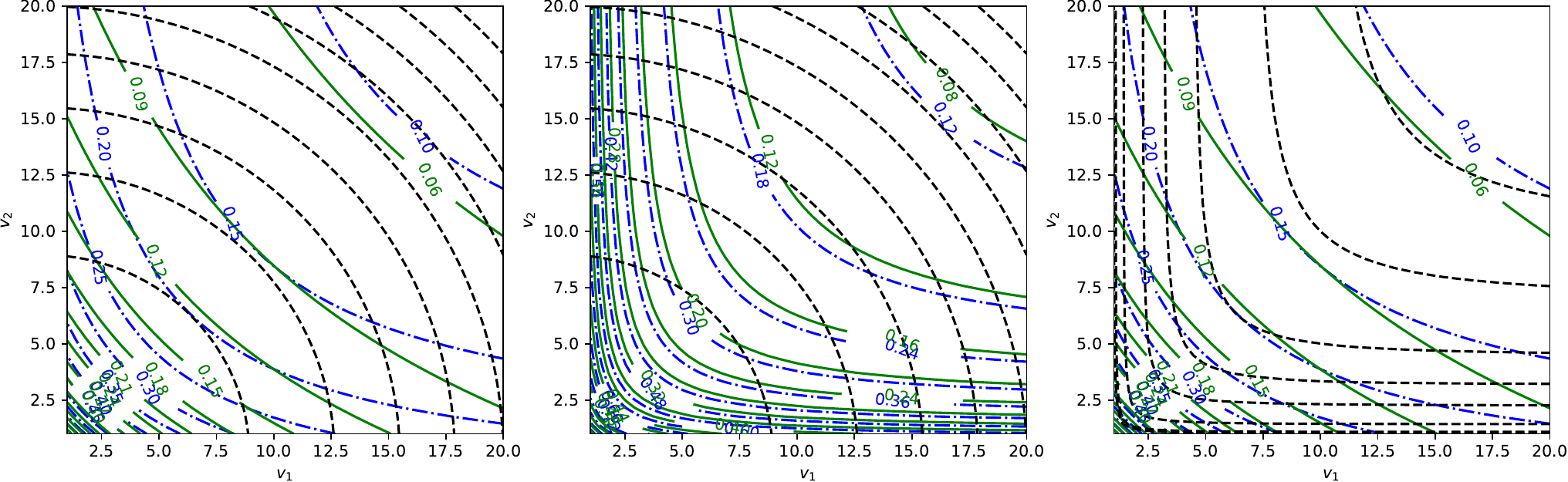}
    \caption{Sketch of the design problem~\eqref{eq:doe-incremental-reduced} for $n=2$ points. Level lines of the objective $E(v)$ without gradient data are drawn by blue lines, whereas those of the budget constraint are indicated by dashed lines. The gradient-based version of $E(v)$ is drawn by solid green lines.  \emph{Left:} For $s>1$ there is a unique non-sparse solution. \emph{Middle:} A smaller correlation length makes sparsity even less likely. \emph{Right:} For $s<1$, the admissible sets are non-convex, and we may expect multiple local sparse minimizers.}
    \label{fig:design-sketch}
\end{figure}
The convexity of the admissible set $\{\tilde v\in\R^{n+j}_+\mid  W(\tilde v)\le \Delta W + W(v)\}$ depends on the exponent $s$  in the work model~\eqref{eq:work-model}. Clearly, for $s\ge 1$, $W$ is convex, whereas for $s<1$ it is in general non-convex, not even quasi-convex. In combination with Thm.~\ref{th:convex-objective}, we obtain the following result.

\begin{corollary}
    For exponents $s\ge 1$, the tolerance design problem~\eqref{eq:doe-incremental-reduced} is convex. 
\end{corollary}

Finite elements of order $r$ in $l$ dimensions with an optimal solver yield $s = l/(2r)$. Consequently, for linear finite elements in two or three space dimensions, any minimizer is a global one, most often unique and non-sparse. In contrast,  high order finite elements with $r\ge 2$ lead to $s<1$ and non-convex admissible sets. Their pronounced corners on the coordinate axes make the sparsity of a minimizer likely, see Fig.~\ref{fig:design-sketch} right. This agrees with intuition: if increasing the accuracy at a specific sample point is computationally expensive, it is advantageous to distribute the work on a lower accuracy level to several points. If increasing the accuracy is cheap, then it is often better to increase the accuracy of a single point, that, to some extent, shares its increased accuracy in a certain neighborhood.
Level lines of the gradient-enhanced objective $E(v)$ are also shown (green). A smaller error can be achieved with the same amount of computational work. Note that gradient data does not affect the convexity of the problem.

While in the convex case $s\ge 1$ the optimization is straightforward with any nonlinear programming solver, the non-convex case is more difficult. Fortunately, guaranteed global optimality is in practice not necessary. The expected sparsity structure suggests a particular heuristic approach: For $i=1,\dots,n+j$ consider $\tilde v=v+ae_i$ with $a>0$ such that $W(v)=\Delta W + W(\mathcal{D})$, i.e. the accuracy of only a single point is improved, and select the design $v$ with smallest objective. If this satisfies the necessary first-order conditions, accept it as solution. Otherwise, perform a local minimization starting from this point.

\section{Numerical examples}\label{sec:experiments}
We present two numerical examples, a low-dimensional one with simple model function $y$, and a PDE model. We compare the results of the adaptive phase with and without gradient data.

\subsection{Analytical example}
As a model $y$ we consider the rotated parabolic cylinder, i.e
\[
    y_\phi(p) = \left ( \cos(\phi)(p_1+p_2)+ \sin(\phi)(p_2-p_1) \right )^2 \quad \text{for $p \in \mathcal{X} = [0,2]^2$}, \; \phi \in \mathbb{R}_{> 0}.
\]
We acquire $m=3$ measurements for $\phi \in \{0,2,4\}$, assume different accuracies of these independent measurements, and a likelihood $\Sigma_{L} = 10^{-2}\mathrm{diag}(1,0.1,1)$.
We start with an initial design of seven evaluation points, see Fig.~\ref{fig:2D_adaptivephase}, and an evaluation variance of $\sigma^2 = 0.1$. We use the work model with $s=1/2$. Evaluation of the error quantity $E$ from~\eqref{eq:error-model} is performed by Monte Carlo integration.
To determine candidate points for inclusion in the evaluation set $X$, we sample the acquisition function~\eqref{eq:utility} for $\beta\in\{0,1\}$. Only the point with the largest value is included, and gradient information included by setting $\tol'=\tol$ if $\beta=1$ obtained a higher value. The adaptive phase is terminated if the error $E$ drops below the desired tolerance $\mathrm{TOL} = 10^{-2}$.

\textbf{Resulting designs.} Fig.~\ref{fig:2D_adaptivephase} illustrates the resulting design with (right) and without (left) gradient data. Black dots are evaluation points without gradient data, while red triangles mark points with gradient information. The size of each marker corresponds to the evaluation accuracy, with larger markers indicating higher accuracy. Red dots represent the initial design, while the color mapping represents the estimated error. The gradient-based algorithm required  24 points, 11 of them with gradient information, each with different evaluation accuracy. The purely value-based design required 30 points.
\begin{figure}[htb!]
    \begin{subfigure}[]{0.46\linewidth}
      \includegraphics[width=\textwidth]{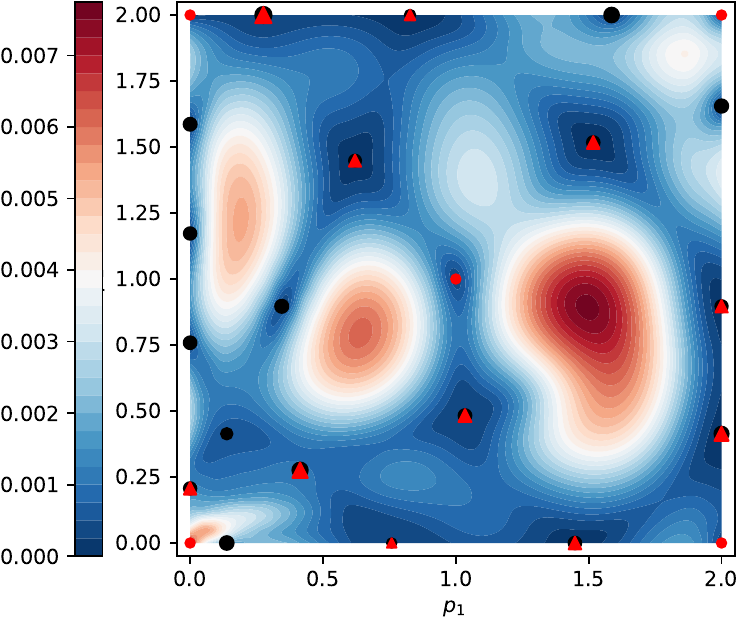}
    \end{subfigure}
    \hfill
    \begin{subfigure}[]{0.52\linewidth}
      \includegraphics[width=\textwidth]{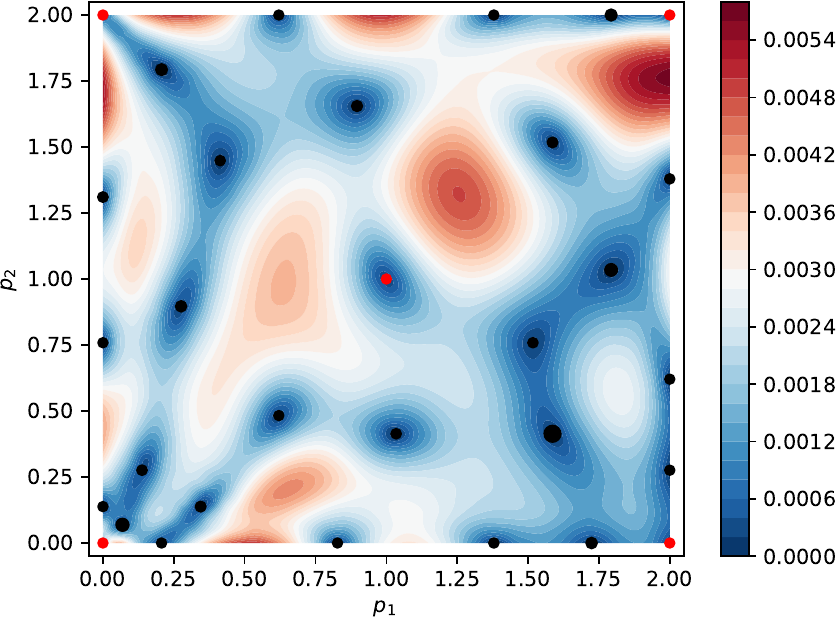}
    \end{subfigure}
    \caption{Contour plot of the local error density. Red points show the initial design. Adaptively added data points are indicated by black dots. Small markers indicate low accuracy. Gradient information is indicated by red triangles.
    The color mapping shows the estimated local reconstruction error evaluated on a dense grid of $10^3$ points. The designs were obtained using an incremental budget of $\Delta W = 10^{4}$ with a desired tolerance of $\mathrm{TOL} = 0.01$.}
    \label{fig:2D_adaptivephase}
\end{figure}

\textbf{Convergence.} Fig.~\ref{fig:123} presents the error against the computational work with different incremental budgets. On the left, we compare designs with gradient data (solid lines) and without (dashed lines). On the right, we compare adaptive designs with gradient data with a position-adaptive design using uniform tolerances for values and gradients.
\begin{figure}
    \begin{subfigure}[]{0.5\linewidth}
      \includegraphics[width=\textwidth]{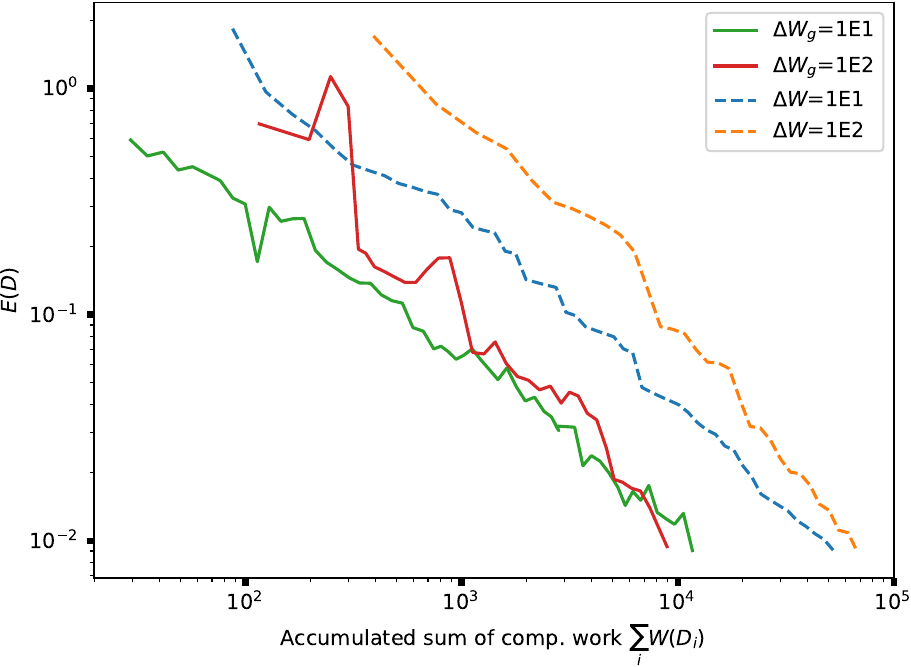}
    \end{subfigure}
    \hfill
    \begin{subfigure}[]{0.5\linewidth}
      \includegraphics[width=\textwidth]{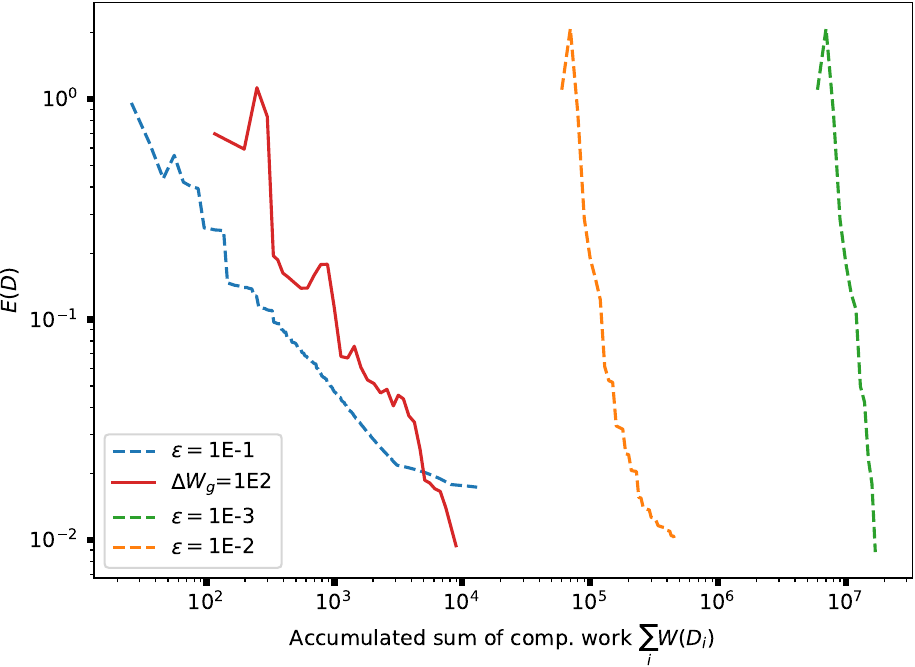}
    \end{subfigure}
        \caption{Estimated global error $E(\mathcal{D})$ versus accumulated computational work in GEGPR surrogates. \textit{Left}: $E(\mathcal{D})$ for different incremental work $\Delta W$. Solid lines with gradient data. \textit{Right}: Different uniform tolerances in position-adaptive designs compared with different curves for $\Delta W_g = 100$. }
    \label{fig:123}
\end{figure}
Including derivative information is clearly more efficient, by roughly one order of magnitude, but also leads to pronounced non-monotone convergence. This is likely an effect of inexact gradient data affecting  hyperparameter optimization. In contrast, value-based designs exhibit a mostly monotonic behavior. Also apparent is that smaller incremental budgets are more efficient, in particular for low desired accuracy.
On the right, the behaviour of uniform tolerance designs is shown. For large tolerances ($\tol=10^{-1}$), the computational work is small, but the desired accuracy is not reached. Smaller tolerances ($\tol\le 10^{-2}$) achieve the desired accuracy, but at a much higher cost, roughly two orders of magnitude and more.

\textbf{Reliability of local error estimates.} The error model~\eqref{eq:error-model} is coarse and may not capture the actual error in identified parameters correctly. It is affected both by linearization error in estimated error transport and the GPR error estimate, which in turn depends also on the hyperparameter optimization. We compare the estimated global error $E(\mathcal{D})$ to the actually obtained errors. For 1600 points $p_i$, sampled randomly from $\mathcal{X}$, we compute the local error estimate $e_i = e_{\mathcal{D}}(p_i)$ from~\eqref{eq:local-error-model}, and compare this to the expected actual error, approximated by the sample mean $\tilde e_i := n_k^{-1} \sum_{k=1}^{n_k} \|p(y(p_i)+\delta_{i,k})-p_\mathcal{D}(y(p_i)+\delta_{i,k})\|$ with  realizations $\delta_{i,k}$ of the measurement error distributed as $\mathcal{N}(0,\Sigma_l)$. 
We construct histograms of the ratio ${e}_i/\tilde{e}_i$. Ratios less than 1 indicate an underestimation of the error, while values greater than 1 indicate an overestimation. 
\begin{figure}
    \centering
 \includegraphics[width=0.8\textwidth]{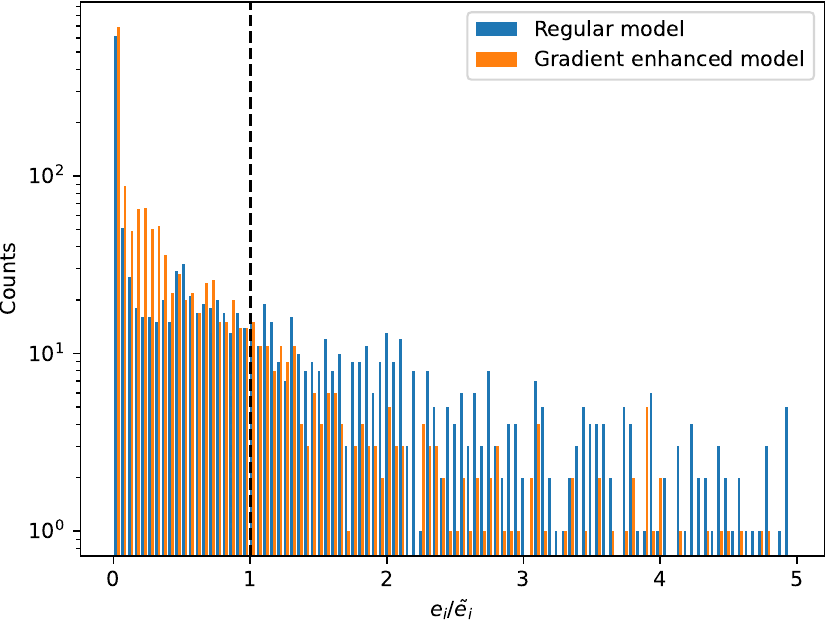}
    \caption{Log-histogram of ${e}_i /\tilde{e}_i$ for value-based GPR (blue) and GEGPR (orange) surrogates.}
    \label{fig:errorhistogram}
\end{figure}
The results in Fig.~\ref{fig:errorhistogram} suggest that value-based and gradient enhanced GPR surrogates behave similar. The coarse a priori error estimate is not strictly reliable, but appears to be reasonably accurate for steering the adaptive design process.

\textbf{Parameter reconstruction.}
For an exemplary parameter reconstruction we assume a true parameter $p = (p_1,p_2) = (1,1.5)$ and exact measurements. Tab.~\ref{tab:results} shows the reconstruction results lying well within the prescribed global tolerance, with comparable errors. 
\begin{table}
    \centering
    \begin{tabular}{cclcl}
    \toprule
     & $p_1$ & $\Delta p_1$ & $p_2$ & $\Delta p_2$  \\ \midrule
    value   & 0.9968&   0.0032  & 1.496  &  0.0047         \\ 
    value+gradient  & 0.9978 & 0.0022 & 1.502  & 0.0023 \\ 
    \bottomrule
    \end{tabular}
    \caption{Reconstruction results using the surrogate model for $\Delta W = 10^4$.}
    \label{tab:results}
\end{table}

\subsection{FEM Example}

Scatterometry is a more complex problem from optical metrology~\cite{hammerschmidt_scat,schneider_scat,farchim_scat}. The aim is to identify geometric parameters in a nano-structured diffraction pattern from the intensity of reflected monochromatic light of different polarizations and incidence angles $\phi,\theta$, see Fig.~\ref{fig:expsetup} and~\cite{Wurm_2011_scat}. A forward model $y(p)$ was developed using JCMsuite\footnote{\url{https://www.jcmwave.com/jcmsuite}} as a solver for the governing Maxwell's equations. The model is parametrized by the geometry of the line grid sample, see Fig.~\ref{fig:expsetup}.

The numerical model maps model parameters $p = \left[ \mathrm{cd}, t, r_{\text{top}}, r_{\text{bot}} \right] \in \mathcal{X} \subset \R^{4}$ to $m=4$  zero'th order scattering intensities $y(p)\in\R^4$, given by incrementing $\theta$ from \SI{5}{\degree} to \SI{11}{\degree} in steps of $\SI{2}{\degree}$, utilizing $P$-polarized light with an angle of incidence $\phi = \SI{0}{\degree}$.
\begin{figure}[b]
    \centering
    \includegraphics[width=0.75\textwidth]{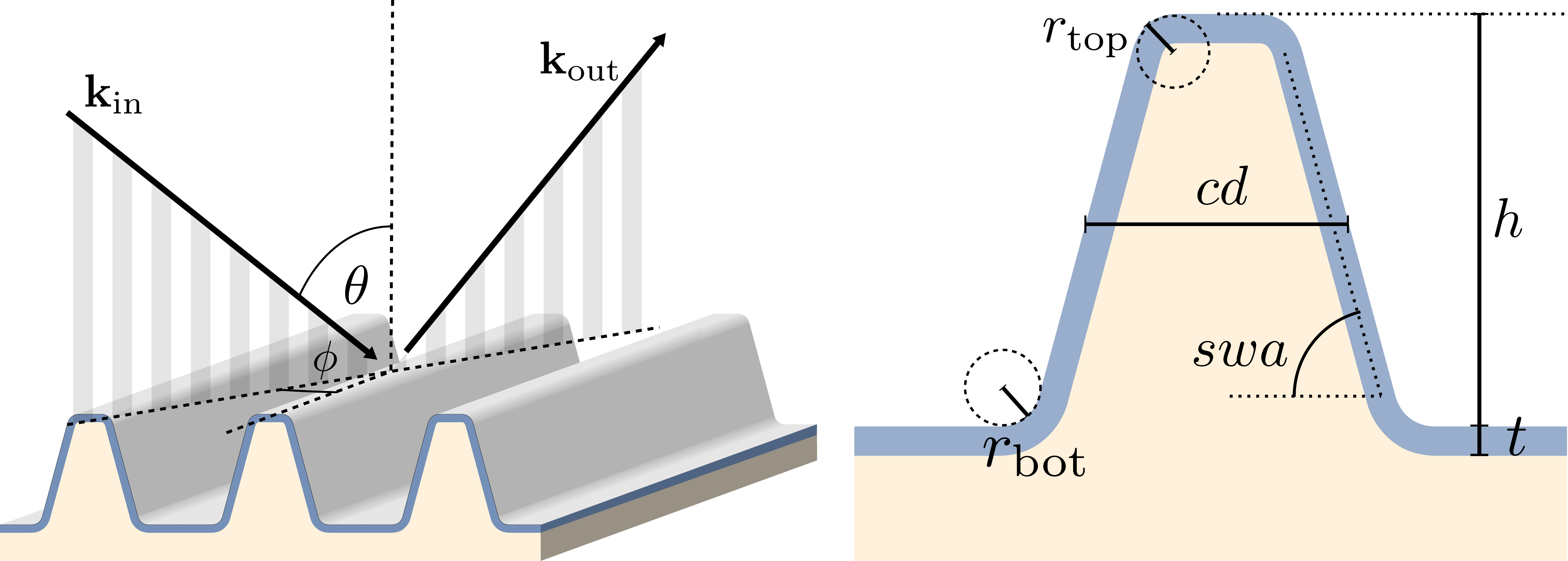}
    \caption{\textit{Left}: Scatterometry setup used for the characterization of periodic nanostructures on surfaces. The incident light is varied at angles $\theta$ and $\phi$ and the diffraction patterns are recorded.  \textit{Right}:  Geometry parametrized in terms of radii $\rt$ and $\rb$, height of the grating, side wall angle (swa), critical dimension (cd), and oxide layer thickness ($t$).}
    \label{fig:expsetup}
\end{figure}

\begin{table}
    \centering
    \begin{tabular}{lccccccc}
    \toprule
        Parameter & $\mathrm{cd}$ & $t$ & $r_{\mathrm{top}}$ & $r_{\mathrm{bot}}$ &\hspace{3em}& $h$ & $\mathrm{swa}$ \\
        \midrule
        range     & $[24,28]\,\si{\nano\meter}$ & $[4,6]\, \si{\nano\meter}$ & $[8,10]\, \si{\nano\meter}$ & $[3,7]\,\si{\nano\meter}$ && $\SI{48.3}{\nano\meter}$ & \SI{87.98}{\degree} \\
        \bottomrule
    \end{tabular}
    \caption{The parameter domain $\mathcal{X}$ for the scatterometry problem.}
    \label{tab:high_dim_domain}
\end{table}

At the beginning, $\mathcal{X}$ is covered with $2^7 = 128$ points from a Sobol sequence. According to the knwon measurement uncertainty, we set \[\Sigma_l = 10^{-2}\text{diag}(8.57,8.56,8.60,8.63)\] and ask for $E(\mathcal{D}) \leq \text{TOL} = 10^{-3}$.
\textbf{Convergence.}
In Fig.~\ref{fig:2D_CompareErrorOverCost} we plot the estimated global error against the computational work.
\begin{figure}
    \begin{subfigure}[]{0.5\linewidth}
      \includegraphics[width=\textwidth]{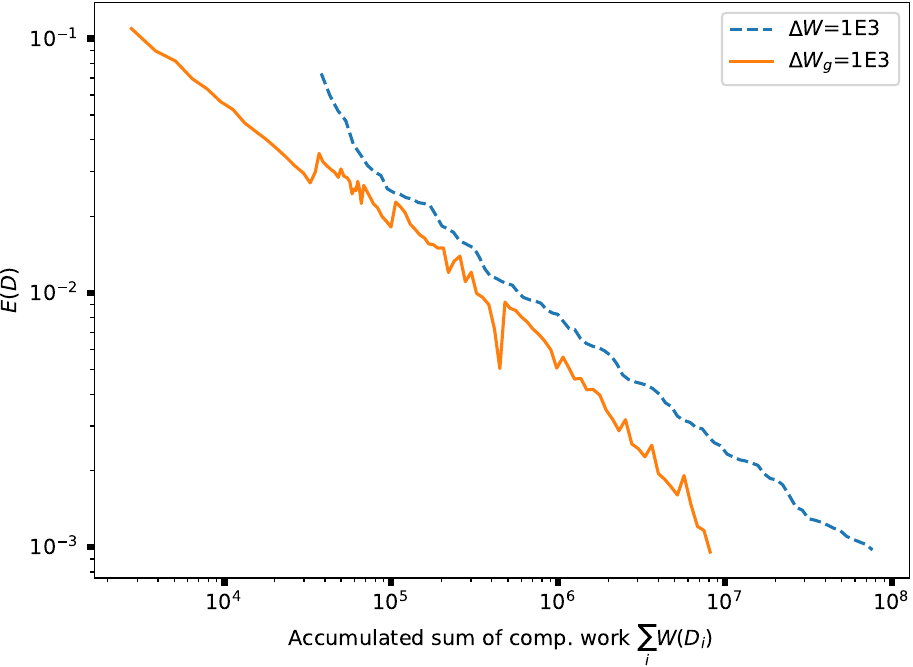}
    \end{subfigure}
    \hfill
    \begin{subfigure}[]{0.5\linewidth}
      \includegraphics[width=\textwidth]{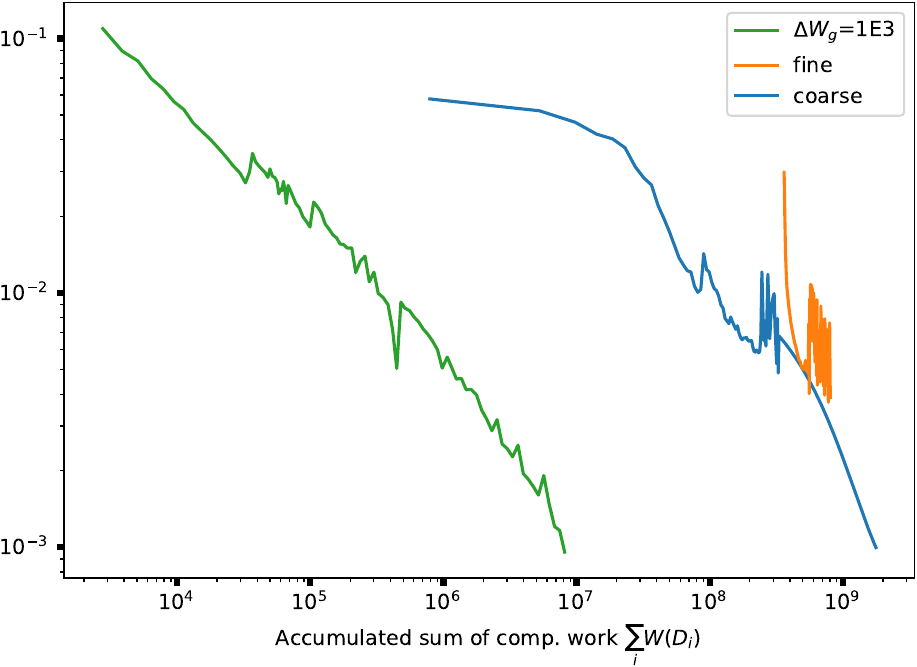}
    \end{subfigure}
    \caption{Estimated global error $E(\mathcal{D})$ versus computational work. \textit{Left}: $E(\mathcal{D})$ for different amounts of incremental work $\Delta W$. The index $g$ indicates a gradient enhanced surrogate model. \textit{Right}: Two different fixed evaluation accuracies within a position-adaptive algorithm compared to a fully adaptive design for $\Delta W_g = 10^3$.}
    \label{fig:2D_CompareErrorOverCost}
\end{figure}
On the left, we a GPR design (dashed) with a GEGPR design (solid) for the same incremental budget $\Delta W$. As in the analytical example, including gradient information improves the efficiency at the cost of a non-monotone convergence behaviour. The improved efficiency is consistent, though quantitatively varying, over a wide range of incremental work budgets, as shown in Tab.~\ref{tab:compwork_compare}.  
\begin{table}[tb]
    \centering
    \begin{tabular}{ccccc}
    \toprule
    $\Delta W$ & $10^3$ & $10^4$ & $10^5$ & $10^6$ \\
    $W/W_g$    & $10$   & $500$  &  $10$  &  $50$ \\
    \bottomrule
    \end{tabular}
    \caption{Performance improvement factor $W/W_g$ due to including gradient information in the adaptive design selection for different incremental budgets.}
    \label{tab:compwork_compare}
\end{table}

In Fig.~\ref{fig:2D_CompareErrorOverCost} right, we compare the presented approach with position-adaptive designs using fixed finite element grids. The coarse grid has a maximum edge length of $h_c = \SI{16}{\nano\meter}$, the fine grid has a maximum edge length of $h_f = \SI{1}{\nano\meter}$, corresponding to FE discretization errors $\epsilon_{c} = 10^{-2}$ and $\epsilon_{f} = 10^{-4}$, respectively. When commencing with the coarse grid, the error reduction in the curve is initially gradual. However, it significantly accelerates beyond a certain threshold, identified as $W \approx 3\cdot 10^7$. Interestingly, beyond this threshold, the presence of pronounced oscillations becomes evident, causing the error to temporarily plateau between $W=2\cdot 10^8$ and $W=4\cdot 10^8$. Subsequently, the curve resumes its descent and eventually achieves the desired tolerance level around $W \approx 2 \cdot 10^9$. Notably, during the analysis, it was observed that within the aforementioned oscillation range, the hyperparameter optimization for a specific parameter component failed to yield satisfactory results. As a result, it was necessary to set this parameter to a default value of $L = 1$, providing a plausible explanation for the observed behavior.

Surprisingly, we encountered difficulties in achieving convergence for the fine grid. Despite an initial rapid decrease in error, the curve exhibited pronounced oscillations across various settings, necessitating the termination of the procedure. Additional tests involving different constant hyperparameters and varied hyperparameter bounds within the optimization failed to yield any improvements. Consequently, further research is required to comprehensively elucidate and resolve this issue.

Thus, in the direct comparison between fully adaptive and semi-adaptive algorithm, we can see that we can save computational work by a factor of $\approx 100$.

\textbf{Reliability of local error estimator.}
As before, we employ $7^4 = 2401$ parameter points to generate the estimated local errors $e_{i}$ and the expected true errors $\tilde{e}_i$ with and without gradient information. A FE simulation on a fine grid with maximum edge length of $\SI{1}{\nano\meter}$ is used as exact forward model. 
\begin{figure}[!htb]
    \centering
 \includegraphics[scale = 0.5]{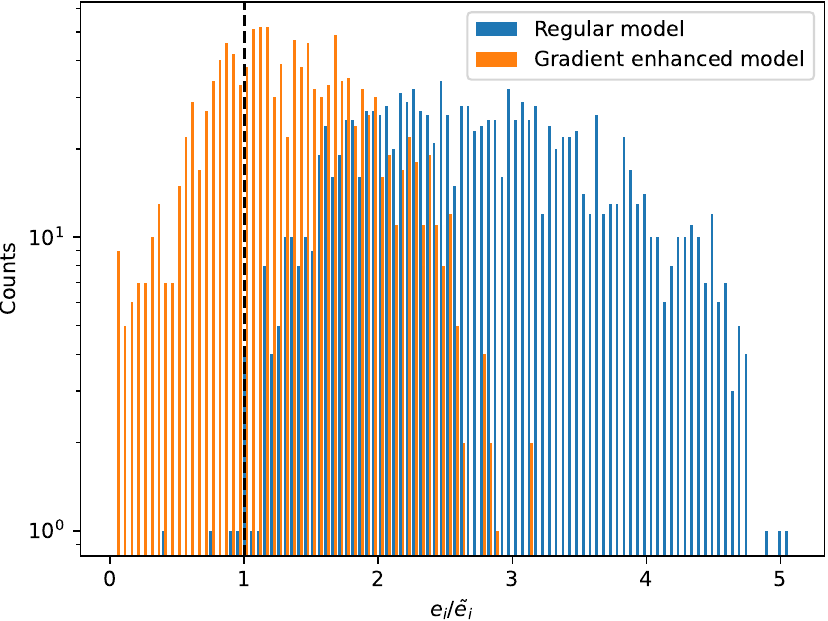}
    \caption{Log-histogram of ${e}_i\cdot \tilde{e}_i^{-1}$. The blue histogram was generated for the regular model, while the orange hisotgram was generated for the gradient enhanced model.}
    \label{fig:errorhistogram}
\end{figure}
In contrast to the analytical example, the GPR and GEGPR surrogates differ slightly, with GPR consistently overestimating the true errors. This suggests that the GPR model should provide smaller actual errors than aimed at. Again, the GEGPR error estimator is not strictly reliable, but apparently sufficiently robust for steering the adaptive design selection process.
 
\textbf{Parameter reconstruction.}
We perform an exemplary parameter reconstruction with true parameters $p_{\mathrm{true}} = (26.0,5.0,10.5,5.0)\, \si{\nano\meter}$ and simulated measurements with artificial measurement noise of size $10^{-3}$. The parameters are recovered within the imposed tolerance, as shown in Tab.~\ref{tab:result_HD}.
\begin{table}
\centering
\begin{tabular}{lcclcc}
\toprule
Parameter  & $p_{\rm CD}$ & $|\Delta p_{\rm CD}|$ & & $p_t$ & $|\Delta p_t|$ \\ 
          \midrule
val        & 26.001  & 1.000E-3 &   & 4.992    & 0.878E-3 \\
val+grad   & 25.999  & 0.983E-4 &   & 4.999 & 0.812E-4    \\ 
\midrule
Parameter  & $p_{\rm top}$ & $|\Delta p_{\rm top}|$ & & $p_{\rm bot}$ & $|\Delta p_{\rm bot}|$ \\ 
          \midrule
val        & 10.502  & 2.00E-3 &   & 4.991    & 9.20E-3   \\
val+grad   & 10.499 & 0.87E-4 &   & 4.999    & 1.54E-4     \\ 
\bottomrule
\end{tabular}
\caption{Reconstruction results for the scatterometry problem.}
\label{tab:result_HD}
\end{table}

\section{Conclusion}

The joint adaptive selection of evaluation positions and evaluation tolerances with a greedy heuristic improves the efficiency of building a GPR surrogate from finite element simulation significantly. Including gradient information enhances this further, if derivatives can be computed cheaply. In numerical experiments, improvement factors between 100 and 1000 have been observed compared to methods relying on only selecting evaluation positions. While the error estimator based on the GPR variance is not strictly reliable, it appears to be sufficiently well-suited for steering the adaptive design selection.

\textbf{Funding.}
This work has been supported by Bundesministerium für Bildung und Forschung -- BMBF, project number 05M20ZAA (siMLopt).

\bibliographystyle{plain}
\bibliography{AdaptiveGPR}

\end{document}